\documentclass[a4paper]{amsart}
 \usepackage[latin1]{inputenc}
\usepackage[dvips]{graphicx}
\usepackage{amsfonts,amsmath,amsthm,amssymb,amscd}
\usepackage{enumerate}

\newcommand{\interior}[1]{\stackrel{\circ}{#1}}

\newcommand{\R}{\mathbb{R}}

\newcommand{\N}{\mathbb{N}}

\newcommand{\G}{\mathcal{G}}

\DeclareMathOperator{\Cat}{Cat}
\DeclareMathOperator{\diam}{diam}

\DeclareMathOperator{\qstraight}{str_{\partial}}

\newcommand{\fl}{\mathit{fl}}

\newtheorem{Def}{Definition}[section]
\newtheorem{Lem}{Lemma}[section]
\newtheorem{Rmk}{Remark}[section]
\newtheorem{Cor}{Corollary}[section]
\newtheorem{Prop}{Proposition}[section]
\newtheorem{Thm}{Theorem}[section]
\newtheorem*{Thm*}{Theorem}
\newtheorem*{Prop*}{Proposition}
\newtheorem*{Cor*}{Corollary}
\newtheorem*{Lem*}{Lemma}

\begin{document}
\title{Typical geodesics on flat surfaces}
\author{Klaus Dankwart}
\begin{abstract}
We investigate typical behavior of geodesics on a closed flat surface $S$ of genus $g\geq 2$. We compare the length quotient of long arcs in the same homotopy class with fixed endpoints for the flat and the hyperbolic metric in the same conformal class. This quotient is asymptotically constant $F$ a.e. We show that $F$ is bounded from below by the inverse of the volume entropy $e(S)$. \\
Moreover, we construct a geodesic flow together with a measure on $S$ which is induced by the Hausdorff measure of the Gromov boundary of the universal cover. Denote by $e(S)$ the volume entropy of $S$ and let $c$ be a compact geodesic arc which connects singularities. We show that a typical geodesic passes through $c$ with frequency that is comparable to $\exp(-e(S)l(c))$. Thus a typical bi-infinite geodesic contains infinitely many singularities, and each geodesic between singularities $c$ appears infinitely often with a frequency proportional to $\exp(-e(S)l(c))$. 
 \end{abstract}
\maketitle
\section{Introduction}
In the study of the Teichm\"{u}ller space of closed Riemann surfaces, half-translation structures play a central role. A half-translation structure on a surface defines a flat metric, see Section \ref{sectflatmetric} for details. \\ 
The geometry of such flat metrics and their relationship to the hyperbolic metric in the same conformal class have been investigated in \cite{Duchin2010, Rafi2007}. We are interested in the asymptotic behavior of geodesics on $S$.\\ 
In the first part of this note we compare the behavior of geodesics for the flat metric $S=(X,d_\fl)$ and the hyperbolic metric $\sigma$ in the same conformal class with respect to the Liouville measure of $\sigma$. 
For a free homotopy classes of closed curves $[\alpha]$ in $X$, denote by $l_*([\alpha]),~*=d_\fl,\sigma$ the infimum of the length of curves for the hyperbolic resp. the flat metric. If $[\alpha]$ is a homotopy class of arcs with fixed endpoints then $l_*([\alpha]),~*=d_\fl,\sigma$ is the infimum of the length of arcs in the homotopy class.\\
Let $g_t:T^1X \to T^1X$ be the geodesic flow for $\sigma$. For $v \in T^1X$ the flow line $g: [i,j] \to T^1X,t \mapsto g_t(v)$ in the unit tangent bundle projects to a local geodesic $c_{i,j}$ on $(X,\sigma)$. Define $F_{i,j}: T^1 X \to \R_+,~v \mapsto l_{\fl}([c_{i,j}])$.\\
 By the subadditive ergodic theorem, $\lim\limits_{T\rightarrow \infty} T^{-1} F_{0,T}$ converge towards the asymptotic length quotient $F$ which is constant a.e.\\
 \cite{Rafi2007} compared the length of simple closed curves for the hyperbolic and the flat metric respectively as follows: The hyperbolic metric on $X$ admits a thick-thin decomposition so that the hyperbolically thin part of the surface is a disjoint union of annuli. Denote by $Y_1, \ldots Y_n$ the components of the thick part. Then there are constants $\lambda_i:=\lambda(Y_i)>0$ so that the following holds:\\ 
Let $[\alpha]$ be a free homotopy class of simple closed curves which can be realized in $Y_i$ and which cannot be homotoped to the boundary. Then the quotient of the flat length and the hyperbolic length of $[\alpha]$ is comparable to $\lambda_i$.\\ 
 Let $\lambda:= \max_i(\lambda_i)$ which might be arbitrary small.\pagebreak
\begin{Thm*}[Theorem \ref{ThmentroF}, Theorem \ref{Thmentrovsrafi}]
The asymptotic length quotient $F$ is related to the volume entropy $e(S)$ and to the factor $\lambda$.
\begin{enumerate}
 \item $F\geq e(S)^{-1}$
 \item There exists some $A>0$ which only depends on the topology of $X$ such that
$$ F\geq A\cdot\lambda.$$ 
\end{enumerate}
\end{Thm*}
In the second part of this note we investigate the typical behavior of a geodesic on a flat surface $S$. Geodesics which to do not pass through singular points have been investigated extensively, see i.e. \cite{Masur1982, Masur1992, Cheun2003,Masur2006, Cheung2010}. Here we allow that geodesics may pass through singular points and change direction which means that a geodesic arc might contain transverse subarcs. As each locally geodesic segment which terminates at a singular point admits a one-parameter family of geodesic extensions, a geodesic flow can not be defined infinitesimally. As in \cite{Gromov1987}, we use the space of all parametrized bi-infinite unit speed geodesics $\G \tilde{S}$ as the unit tangent bundle on the isometric universal cover $\pi: \tilde{S} \to S$.\\ 
To define an appropriate measure for the geodesic flow in terms of the flat metric we follow the idea of the Hopf parametrization \cite{Hopf1971}:
Fix a point $\tilde{x} \in \tilde{S}$. $\tilde{S}$ admits a metric boundary $(\partial \tilde{S},d_{\infty,\tilde{x}})$, the so-called Gromov boundary, and $\G \tilde{S}$ naturally projects onto $\partial \tilde{S}^2-\triangle$. In other words, $(\partial \tilde{S}^2-\triangle) \times \R$ is non-canonically a quotient of $\G \tilde{S}$.\\ 
The Hausdorff measure $\nu_{\tilde{x}}$ of $(\partial \tilde{S},d_{\infty,\tilde{x}})$ is positive finite and induces a positive finite quotient measure $\mu$ on the quotient space $\G \tilde{S}/\pi_1(S)$ which is invariant under the geodesic flow. We show that $g_t$ acts ergodically with respect to $\mu$.\\ 
We use this construction to investigate typical behavior of geodesics. Denote by $e(S)$ the volume entropy of the flat surface which was estimated in \cite{Dankwart2011}. Let $c$ be a locally geodesic arc on $S$ which might contain singular points. Extend $c$ as much as possible in positive and negative direction with the property that the extension is unique. Let $c_{ext}$ be the extended arc which might be infinite. We estimate the frequency $\varLambda$ of a $\mu$-typical geodesic passing through $c$. 
\begin{Thm*}[Theorem \ref{Thmgenericbehaviour}]
For a closed flat surface $S$ there is a constant $C(S)>0$ such that the following holds:\\ 
For any compact locally geodesic arc $c$ a typical geodesic passes through $c$ with a frequency $\varLambda(c)$ which is bounded from above and below by 
$$C(S)^{-1}exp(-e(S)l(c_{ext}))\leq \varLambda(c) \leq C(S) exp(-e(S)l(c_{ext})).$$ 
\end{Thm*}
So we obtain a concrete description of typical geodesics as concatenations of geodesic segments which connect singular points. Each such segment appears infinitely often in each typical geodesic and the frequency decrease exponentially in the length of the segment. As a corollary, geodesics which omit the singularities are of measure zero.\\
The paper is organized as follows. In the next section we recall the main facts about $\delta$-hyperbolic spaces and flat surfaces which are needed in the sequel. 
In section \ref{sectcomhyp} we investigate the asymptotic length comparison of geodesic in the flat and hyperbolic metric. In section \ref{sectasymprays} we collect technical lemmas concerning geodesic rays on flat surfaces. In section \ref{sectgeodflow} we estimate ergodic behavior of geodesics on flat surfaces. \\
 \textbf{Acknowledgement} This paper is part of my Dissertation. I would like to thank my advisor Ursula Hamenst\"{a}dt for her support and for raising the questions. I would also like to thank Sebastian Hensel, Emanuel Nipper and Jon Chaika for many discussion. Parts of this work were done during my stay at the Universit\'{e} Paul Cezanne. I would like to thank Pascal Hubert for inviting me and for many advices. This research was supported by Bonn International Graduate school in Mathematics. 

\section{Preliminaries about flat surfaces and spaces of non-positive curvature}
\subsection{Gromov hyperbolic spaces and Patterson Sullivan measures}\label{sectgrhyp}
We recall the standard facts about proper $\delta$-hyperbolic spaces. For details we refer to \cite[ Chapter III] {BridsH1999} .\\ 
\textbf{Convention:} Any metric space $X$ is assumed to be complete, proper and geodesic.\\
A mapping $f:X\to Y$ between metric spaces $X$ is an $L$-\textit{quasi-isometric embedding} if and only if
$$L^{-1}d_X(x_1,x_2)-L\leq d_Y(f(x_1),f(x_2))\leq Ld_X(x_1,x_2)+L,~\forall x_1,x_2 \in X.$$ 
An $L$-quasi-isometric embedding is an $L$-\textit{quasi-isometry} if $\sup\limits_{y \in Y}d_Y(f(X),y)<\infty$. An $L$-\textit{quasi-geodesic} $c:I \to X$ is an $L$-quasi-isometric embedding of a real line segment $I\subset\R$ to $X$.\\
A metric space $X$ is $\delta$-\textit{hyperbolic} if every geodesic triangle in $X$ with sides $a, b, c$ is $\delta$-\textit{slim}: The side $a$ is contained in the $\delta$-neighborhood of $b \cup c$.\\
A $\delta$-hyperbolic space $X$ admits a metric boundary which is defined as follows. Fix a point $p\in X$ and for any two points $x_1,x_2\in X$ we define the \textit{Gromov product} 
$$(x_1 \cdot x_2)_p :=\frac{1}{2}(d(x_1,p)+d(x_2,p)-d(x_1,x_2)).$$ 
We call a sequence $x_i$ \textit{admissible} if $(x_i\cdot x_j)_p\rightarrow \infty$. We define two admissible sequences
 $x_i,y_i\subset X$ to be equivalent if $(x_i \cdot y_i)_p \rightarrow \infty$. Since $X$ is $\delta$-hyperbolic, this defines an equivalence relation. The Gromov boundary $\partial X$ of $X$ is then the set of equivalence classes of admissible sequences.\\ 
The space $\overline{X}=\partial X \cup X$ can be equipped with a topology to yield a \textit{compactification} of $X$:
For a point $p\in X$ and $U\subset X$ denote by $sh_p(U) \subset \overline{X} $ the $U$-\textit{shadow } for $p$: The set of all points $x \in X$ such that there is a geodesic connecting $p$ and $x$ which intersects $U$, together with all boundary points $\eta \in \partial X$ such that there is a sequence $\{y_i\} \in \eta$ with $y_i \in sh_x(U)$ for almost all $i$. The \textit{boundary shadow } is defined as $\partial sh_x(U):=\partial X \cap sh_x(U) $.\\
The shadows $sh_p(U)$ for all $p$ and for all open sets $U$, together with all open balls in $X$ form the basis for the topology on $\overline{X}$.\\
The \textit{Gromov product on the boundary} is defined by 
$$(\eta \cdot\zeta)_p=\sup\{\liminf\limits_{i,j}(x_i\cdot y_j)_p~|~\{x_i\} \in \eta ,~\{y_j\} \in \zeta\}. $$
\begin{Lem} \label{lemgrmetric}
 For a $\delta$-hyperbolic space $X$ let $\delta_{\inf}$ be the infimum of all its hyperbolic constants and let $\xi:= \xi(\delta_{\inf})=2^{\frac{1}{2\delta_{\inf}}}$. Then, for any point $p\in X$ there is a metric $d_{\infty,p}$ on $\partial X$ and a constant $\epsilon(\delta_{\inf}) < 1$ which satisfies:
$$(1-\epsilon(\delta_{\inf}))\xi^{-(\eta\cdot \zeta)_p}\leq d_{p,\infty}(\eta, \zeta)\leq \xi^{-(\eta\cdot \zeta)_p}.)$$ 
\end{Lem}
\begin{proof}
\cite[Proposition 3.21] {BridsH1999}
\end{proof}
$d_{\infty,p}$ is a \textit{Gromov metric} on the Gromov boundary $\partial X$. A quasi-isometry between $\delta$-hyperbolic spaces extends to a homeomorphism between the boundaries.\\ 
Denote by $d_{\infty}$ the bilipschitz equivalence class of the Gromov metrics $d_{\infty,p},~ p\in X$.\\
Any bi-infinite $L$-quasi-geodesic converges to two distinct boundary points, and between any two distinct boundary points there is a connecting bi-infinite geodesic.\\
\textbf{Notation:} If $X$ is a $\delta$-hyperbolic $\Cat(0)$ space denote by $[x,y],~x,y\in \overline{X}$, a geodesic connecting $x$ with $y$. For $x\in X,~y\in \overline{X}$, $[x,y]$ is unique up to reparametrization.
\begin{Lem}
There is a function $H(L,\delta)>0$ such that for any $\delta$-hyperbolic space $X$ and for any two $L$-quasi-geodesics $c,c'$ in $X$ with the same endpoints in $\overline{X}$ the Hausdorff distance of $c,c'$ is at most $H(L,\delta)$.
\end{Lem}
\begin{proof}
This follows from \cite[III 1.7]{BridsH1999} and \cite[I Proposition 3.2]{CoornP1993}
\end{proof}
For a $\delta$-hyperbolic space $X$ and $\Gamma$ a group of isometries acting properly discontinuously, freely and cocompactly on $X$, fix $p\in X$ and $\ell$ a $\Gamma$-invariant non-zero Radon-measure on $X$. Denote by $B_p(R)$ the open $R$-ball about $p \in X$. The \textit{volume entropy} is defined as
$$e(X,\Gamma):= \limsup\limits_{R \to \infty}\frac{\log\left(\ell\left(B_p(R)\right)\right)}{R}.$$
\textbf{Convention:} By entropy we mean volume entropy. Here $X$ is the isometric universal cover of $X/\Gamma$, so we abbreviate $e(X/\Gamma):=e(X,\Gamma)$. We always assume that the group $\Gamma$ acts properly discontinuously, freely and cocompactly by isometries on $X$. 

We recall the construction of Patterson-Sullivan measures. Rigorous computations can be found in \cite[Section 1-3]{Sulli1979}, \cite[Section 4-8]{Coorn1993}.\\
We define the Poincar\'{e} series
$$g_s(p):=\sum\limits_{y \in \Gamma p} \exp(-sd(p,y)).$$
\begin{Prop}
For a $\delta$-hyperbolic space $X$ and $\Gamma$ acting cocompactly on $X$, $g_s(p)$ is finite if and only if $s>e(X/ \Gamma)$.
 \end{Prop}
For $s>e(X/ \Gamma),x \in X$ one defines the Radon measure 
$$\nu_{s,x}:= \frac{1}{g_s(p)}\sum\limits_{y \in \Gamma p} \exp(-sd(x,y))\delta_{y} $$
on $\overline{X}$ where $\delta_y$ is the Dirac measure. For $s_i \searrow e(X/ \Gamma)$, $\nu_{s_i,x}$ converges towards a Radon measure $\nu_x$ which is again finite and supported on the whole boundary. It satisfies 
$$\gamma^*\nu_{\gamma(x)}=\nu_{x},~\forall \gamma\in \Gamma.$$
\begin{Thm} \label{thmhdimentro}
Let $X$ be a $\delta$-hyperbolic $\Cat(0)$-space and let $d_{\infty,x}$ be the Gromov metric on the boundary with respect to some base point $p\in X$. Assume that $\Gamma$ acts cocompactly on $X$. \\ 
Then the Hausdorff dimension of the boundary coincides with $\frac{e(X/ \Gamma)}{\log(\xi(\delta_{\inf}))}$.\\
Furthermore, the Hausdorff measure of $d_{\infty,x}$ coincides with $\nu_x$, up to a multiplicative constant. 
\end{Thm}
\begin{proof}
 We refer to \cite{Coorn1993}.
\end{proof}
For simplicity we extend $\nu_x$ to a complete measure and assume that we always take the completion of any measure instead of the measure itself. 
\subsubsection{Radon-Nikodym derivative}
The family of measures $\nu_x,~x\in X$, is contained in the same measure class. In the later context we need to estimate the difference of any two measures.\\
For a $\Cat(0)$-space $X$ and for $\eta \in \partial X$, the \textit{Busemann distance} is defined as $b(x,y, \eta):=\lim\limits_{t\rightarrow \infty} t- d([x, \eta](t),y),~x,y \in X$ 
\begin{Lem}\label{Lemradder}
For a $\delta$-hyperbolic $\Cat(0)$-space $X$ there exists a constant $C(\delta)$, only depending on $\delta$ so that for all $x,y \in X,~ \eta \in \partial X$ the Radon-Nikodym derivative can be estimated by
$$\exp(-e(X/ \Gamma)(b(x,y, \eta)+C(\delta)))\leq \frac{d\nu_{x}}{d\nu_{y}}(\eta)\leq \exp(-e(X/ \Gamma)(b(x,y, \eta)-C(\delta))). $$
If $y$ is contained in the connecting geodesic $[x, \eta]$ it follows:
$$\frac{d\nu_{x}}{d\nu_{y}}(\eta)=\exp(-e(X/\Gamma)b(x,y, \eta))=\exp(-e(X/\Gamma)d(x,y)).$$ 
\end{Lem}
\begin{proof}
This is shown in \cite{Coorn1993}.
\end{proof}
Denote by $\diam(X/\Gamma):=\sup\limits_{x,y \in X} d(x, \Gamma y)$ the diameter of $ X/\Gamma $.
\begin{Cor} \label{Corestrad}
For, $x,y \in X$, the Radon-Nikodym derivative is bounded by
$$\exp(-e(X/ \Gamma)d(x,y))\leq \frac{d\nu_{x}}{d\nu_{y}}\leq \exp(e(X/ \Gamma)d(x,y)).$$
 Furthermore, the measure of the whole boundary is bounded by
$$\exp(-e(X/ \Gamma)\diam)\leq \nu_x(\partial X)\leq \exp(e(X/ \Gamma)\diam(X/\Gamma)).$$
 \end{Cor}
\begin{proof}
 The first inequality follows from Lemma \ref{Lemradder} and the second is a consequence of the fact that $\nu_{\gamma(x)}(\gamma(A))=\nu_{x}(A)$ and there exists some $x$ so that $\nu_{x}\left(\overline{X}\right)=1$. 
\end{proof}

\subsection{Geometry of flat surfaces}\label{sectflatmetric}
We introduce the geometry of flat surfaces and refer to \cite{Minsk1992, Streb1984}.\\
\textbf{Convention}: Any closed oriented surface is assumed to be of genus $g\geq 2$.\\
A \textit{half-translation structure} on a closed oriented surface $X$ is a choice of charts such that, away from a finite set of points $\Sigma$, the transition functions are half-translations. The pull-back of the flat metric in each chart gives a flat metric $d_\fl$ on $X-\Sigma$. We require that $d_\fl$ extends to a singular cone metric on $X$ with cone angle $k\pi$ in each point $\varsigma \in \Sigma$ where $k\geq 3$ is an integer. Then $S=(X,d_\fl)$ is a \textit{flat surface}.\\
A \textit{straight line segment} on $S-\Sigma$ is defined as the pull-back of a straight line segment on $\R^2$ in each chart. A \textit{saddle connection} is a straight line segment which emanates from one singularity and ends at another.\\
The flat metric admits an unoriented \textit{flat angle} as follows. A \textit{standard} neighborhood $U$ of a point $p \in S$ is isometric to a cone around $p$. On the boundary circle of $U$ we choose an orientation. Let $c_1,c_2$ be straight line segments, issuing from $x$. The complement $U -c_1 \cap c_2$ consists of two connected components $U_1,U_2$ which are isometric to euclidean circle sectors with angle $\vartheta_i,~ i=1,2$, possibly greater than $2\pi$. Choose $U_1$ such that the arc on $\partial U$ which connects $c_1$ with $c_2$ in direction of the boundary orientation is on the boundary of $U_1$. Then $\angle_p(c_1,c_2)$ is the sector angle $\vartheta_1$. 
\begin{Lem} 
A path $c:[0,T]\to S$ is a local geodesic if and only if it is continuous and a sequence of straight lines segments outside $\Sigma$. At the singularities $\varsigma=c(t)$ the consecutive line segments make a flat angle at least $\pi$ with respect to both boundary orientations. \\
A compact local geodesic $c:[0,T] \to S$ can be extended to a locally geodesic line $c':\R \to S$.
\end{Lem}
\begin{proof}
 We refer to \cite[Theorem 8.1]{Streb1984}.
\end{proof}
\begin{Lem}\label{lemscdense}
 Let $S$ be a closed flat surface and $x \in S$ be any point. Fix an outgoing direction $\theta$ at $x$ and the flat angle $\angle_x$ with respect to a choice of orientation. For any $\epsilon$ there exists a geodesic $[x,\varsigma]$ which emanates from $x$ and ends at a singularity $\varsigma$ so that $\angle_x([x,\varsigma], \theta)\leq \epsilon$. 
\end{Lem}
\begin{proof}
We refer to \cite[Proposition 3.1]{Vorob1996}.
 \end{proof}
One of the main concepts is the \textit{Gauss-Bonnet formula}, see \cite{Hubba2006}. Let $P$ be a compact flat surface with piecewise geodesic boundary and denote by $\chi(P)$ the Euler characteristic of $P$. For $x \in \interior{P}$, define $\vartheta(x)$ to be the cone angle of $x$ in $P$. For $x\in \partial P$, $\vartheta(x)$ is the cone angle at $x$ inside $P$. Then 
$$2\pi\chi(P)=\sum_{x \in \interior{P}} (2\pi-\vartheta(x))+\sum \limits_{x\in \partial P} (\pi-\vartheta(x)).$$ 
As a corollary, a flat surface does not contain geodesic bigons.
\begin{Prop}
 In any homotopy class of arcs with fixed endpoints on a closed flat surface there exists a unique local geodesic which is length-minimizing.\\ 
Moreover, in any free homotopy class of closed curves there is a length-minimizing locally geodesic representative.
\end{Prop}
\begin{proof}
In both cases, the existence follows from a standard Arzel\`{a}-Ascoli argument. The uniqueness follows from the absence of geodesic bigons. 
\end{proof}
A \textit{flat cylinder} of height $h$ and circumference $c$ on a flat surface $S$ is an isometric embedding of $[0,c] \times (0,h)/\sim, (0,t)\sim (c,t) $ into $S$. A flat cylinder is \textit{maximal} if it cannot be extended. \\ 
Like in hyperbolic geometry, closed local geodesics on flat surfaces do not have arbitrary intersections.
\begin{Lem} \label{Lemintnumbergeod}
 Let $\alpha, \beta$ be closed curves on a flat surface $S$ and $\alpha_{\fl}, \beta_\fl$ be a choice of locally geodesic representatives in the free homotopy class. If the number of intersection points of $\alpha_\fl$, $\beta_\fl$ is bigger than the geometric intersection number $i([\alpha],[\beta])$, then the local geodesics $\alpha_\fl$ and $\beta_\fl$ share some arcs which start and end at singularities. The arcs might be degenerated to singular points.
\end{Lem}
 \begin{proof}
If $\alpha_\fl$ and $\beta_\fl$ have more points in common than $i(\alpha, \beta)$, by absence of geodesic bigons, $\alpha_\fl$,$\beta_\fl$ share some arc. As local geodesics in flat surfaces are straight line segments outside the singularities, two local geodesics having some arc in common can ones drift apart at singularities. So they can ones share arcs with singularities as start- and endpoints. 
\end{proof}
\begin{Prop}\label{Propjoingeod}
For a closed flat surface $S$ there is a constant $C_l(S)>0$ so that following holds:\\
For any two parametrized geodesics $c,c'$ on $S$, so that $c$ ends at a singularity and $c'$ issues from some singularity there is a local geodesic $g$ which first passes through $c$ and eventually passes through $c'$ and which is of length at most $C_l(S)+l(c)+l(c')$.
\end{Prop}
\begin{proof}
 We refer to \cite{Dankwart2011}.
\end{proof}
 
\section{Asymptotic comparison of hyperbolic and flat geodesics}\label{sectcomhyp}
The isometric universal cover $\pi: \tilde{S} \to S$ of a closed flat surface $S$ is a topological disc with $\Cat (0)$ metric. By the \v{S}varc-Milnor Lemma, $\tilde{S}$ is quasi-isometric to the Poincar\'{e} disc and so $\delta$-hyperbolic for some $\delta$, see \cite{Dankwart2011} for more precise estimates. Denote by $\Gamma$ the group of deck transformations.\\
For a closed flat surface $S=(X,d_\fl)$ let $\sigma$ be the hyperbolic metric in the same conformal class as $d_\fl$. Recall that the length of a homotopy class of arcs with fixed endpoints $l_*([c]),~*=d_\fl,\sigma$ is defined as the length of the shortest arc for the metric.\\ 
To investigate the relationship between the two metrics, we compare the geodesic flows. 
 Let $\tau: T^1X \to X$ be the unit tangent bundle of $X$ and let $\ell_*, *=d_\fl,\sigma$ be the Lebesgue measure on $X$ defined by the flat resp. hyperbolic metric. Denote by $\mathfrak{m}$ the Liouville measure for the hyperbolic metric on the unit tangent bundle and let $\mathfrak{m}_x, x \in X$ be the induced measure of the fiber $T^1_x X\cong S^1$.\\
The measures are defined in the same way on the universal cover $\pi:\tilde{X} \to X$ endowed with the lifted structure, i.e. the unit tangent bundle $\tau:T^1 \tilde{X} \to \tilde{X}$, the hyperbolic or flat metric and the covering map $\pi$ naturally extends to $\pi:T^1\tilde{X} \to T^1X$.\\
The geodesic flow $g_t:T^1X\to T^1X$ for the hyperbolic metric $\sigma$ on $T^1X$ acts ergodically with respect to $\mathfrak{m}$.\\ 
 To each point $v \in T^1X$ and to each $i, j \in \R $, associate the a local geodesic $c_{i,j}:[i,j] \to X, c_{i,j}(t):= \tau(g_{t}(v))$. Define 
$$F_{i,j}: T^1 X \to \R_+: v \mapsto l_\fl([c_{i,j}])$$
to be the flat length of the homotopy class $[c_{i,j}]$.
 By the subadditive ergodic theorem $\lim\limits_{T \rightarrow \infty} 1/TF_{0,T}$ converges to a constant function $F$ with respect to the $L_1$-norm.
\begin{Lem}\label{Lemfactergprocess}
There exists a constant $\lambda>0$ such that for each $s \geq 0, ~v\in T^1X$
$$F_{0,s}\leq F_{0,s+t}, \forall t \geq \lambda.$$
\end{Lem}
\begin{proof}
Since the flat and the hyperbolic metric on the universal cover $\tilde{X}$ are $L$-quasi-isometric for some $L$, $L^{-1}|i-j|-L\leq F_{i,j}\leq L|i-j|+L$. As $L$-quasi-geodesics with the same endpoints in $\delta$-hyperbolic spaces have uniformly bounded Hausdorff distance, there is some $\lambda'>0$ that
$$F_{0,t}(v)+ F_{t,t+s}(v) \leq F_{0,s+t}(v) +\lambda',\forall s,t>0, ~ v \in T^1X.$$
For $\lambda:= L(\lambda'+1)$ this proves the claim. 
\end{proof}

\subsection{Comparison with the entropy}
\begin{Thm}\label{ThmentroF}
For a closed flat surface $S=(X,d_\fl)$, the entropy and the constant $F$ are related.
$$e(S) \geq F^{-1}$$
\end{Thm}
\begin{proof}
Assume on the contrary that $e(S) < F^{-1} -3\epsilon$ for some $0<\epsilon < F^{-1}$ and notice that $\frac{1}{F^{-1}-\epsilon}>F$. \\
For $x\in X$ define
$$c_r(x):= \mathfrak{m}_x\left(v \in T^1_{x}X: F_{0,r(F^{-1}-\epsilon)}(v)\geq r\right). $$ 
Since $r^{-1}F_{0,r}$ converges towards $F$ with respect to the $L_1$-norm, for each $\ell_\sigma$-typical point $x$, $c_r(x)$ tends to zero if $r$ tends to infinity.\\
 Fix $\tilde{x}\in \pi^{-1}(x) \subset \tilde{X}$, a preimage of $x$ in the universal cover which is $\ell_\sigma$-typical and let $B^*_{\tilde{x}}(r)$ be the open metric ball with center $\tilde{x}$ and radius $r$ with respect to the metric $*=\sigma,d_\fl$. Notice that 
$$\lim\limits_{r\rightarrow \infty}\frac{\log\Bigl(\ell_\sigma\Bigl(B^\sigma_{\tilde{x}}(r)\Bigr)\Bigr)}{r}=1.$$ 
$\ell_\sigma$ is a $\Gamma$-invariant Radon measure on $\tilde{X}$ and so
$$ e(S)= \limsup\limits_{r \rightarrow \infty}\frac{log\left(\ell_{\sigma}\left(B^{\fl}_{\tilde{x}}(r)\right)\right)}{r}.$$
One estimates $F$ by comparing the volume of metric balls:
\begin{eqnarray*}
&& \ell_{\sigma}\Bigl(B_{\tilde{x}}^\sigma(r(e(S)+\epsilon))\Bigr)-\ell_{\sigma}\Bigl( B_{\tilde{x}}^{\fl}(r)\Bigr)\\
&\leq& \ell_{\sigma}\Bigl(B_{\tilde{x}}^\sigma(r(e(S)+\epsilon))- B_{\tilde{x}}^{\fl}(r)\Bigr)\\
&=& \ell_{\sigma}\Bigl( \left\{\tau(g_t(v))|(v,t)\in T^1_{\tilde{x}}\tilde{X}\times [0, r(e(S)+\epsilon)]:F_{0,t}(\pi(v))\geq r\right\}\Bigr)
\end{eqnarray*}
For $\lambda$ as in Lemma \ref{Lemfactergprocess} let $r$ be so large that $r\epsilon \geq 2\lambda$ and so 
$$F_{0,r(F^{-1}-\epsilon)}(\pi(v)) \geq F_{0,t}(\pi(v)),~\forall v \in T^1\tilde{X},~t\in [0, r(e(S)+\epsilon)].$$
Therefore, we can estimate
\begin{eqnarray*}
&&\ell_{\sigma}\Bigl(B_{\tilde{x}}^\sigma(r(e(S)+\epsilon))- B_{\tilde{x}}^{\fl}(r)\Bigr)\\
&\leq& \ell_{\sigma}\left(\left\{\tau(g_t(v))|(v,t)\in T^1_{\tilde{x}}\tilde{X}\times [0, r(e(S)+\epsilon)]:F_{0,r(F^{-1}-\epsilon)}(\pi(v))\geq r\right\}\right)\\
&=& \ell_{\sigma}\Bigl(B_{\tilde{x}}^\sigma(r(e(S)+\epsilon))\Bigr)c_r(x)
\end{eqnarray*}
In summary:
\begin{eqnarray*}
&&\ell_{\sigma}\Bigl(B_{\tilde{x}}^\sigma(r(e(S)+\epsilon))\Bigr)-\ell_{\sigma}\Bigl(B_{\tilde{x}}^{\fl}(r)\Bigr)\leq \ell_{\sigma}\Bigl(B_{\tilde{x}}^\sigma(r(e(S)+\epsilon))\Bigr)c_r(x)\\
&\Leftrightarrow& \frac{log(1- c_r(x))+\log\Bigl(\ell_{\sigma}\Bigl(B_{\tilde{x}}^\sigma(r(e(S)+\epsilon)\Bigr)\Bigr)}{r}\leq \frac{log\Bigl(\ell_{\sigma}\Bigl(B_{\tilde{x}}^{\fl}(r)\Bigr)\Bigr)}{r}
\end{eqnarray*}
If $r$ tends to infinity, the left term tends to $e(S)+\epsilon$ whereas the right term is bounded from above by $e(S)$, what is a contradiction.
\end{proof}
\subsection{Comparison with the Rafi constant}
We showed that for a closed flat surface $S=(X,d_\fl)$ the entropy bounds the asymptotic quotient of flat and hyperbolic length of geodesic arcs. 
A similar quantity was defined in \cite{Rafi2007}. For a free homotopy class simple of closed curves $[\alpha]$ in $X$ denote:
$$l_*([\alpha]):=\min_{\alpha'\in [\alpha]}l_i(\alpha'),*=d_\fl, \sigma$$
In each free homotopy class of closed curves the minimum is attained by a local geodesic: 
 $$l_*([\alpha])=l_*(\alpha_*),*=d_\fl, \sigma$$
We remark the following standard facts concerning the geometry of a closed hyperbolic surface $(X,\sigma)$ and refer to \cite{Keen1974, BenedP1992, Thurs1980, Rafi2007}. 
\begin{enumerate}[(1)]
\item Around any simple closed local geodesic $\alpha_\sigma$ there exists an equidistant convex collar. The distance from the boundary to the locally geodesic core curve is bounded from below by a function $r(l_{\sigma}(\alpha_\sigma))$ which is independent from the surface $(X, \sigma)$, decreasing and unbounded. 
 \item Let $\epsilon$ be the Margulis-constant which only depends on the topology of $X$. The \textit{thin part} $X_< \subset X$is the set of points with injectivity radius of at most $\epsilon$ which is a disjoint union of convex annuli. The diameter of each connected component of the \textit{thick part} $X_>:=X-X_<$ is bounded from above by a constant which only depends on the topology of $X$.
\item For $L:=\log(\epsilon^{-1})$ and for each component $Y$ of $X_>$ which is not a pair of pants there are two intersecting not freely-homotopic non peripheral simple closed curves in $Y$ whose lengths are bounded from above by $L$.
\end{enumerate}
For a closed flat surface $S=(X,d_\fl)$ and the hyperbolic metric $\sigma$, let $(X_>,X_<)$ be the thick-thin decomposition of $(X, \sigma)$ with respect to the Margulis constant and let $Y$ be a connected component of $X_>$.\\ 
 \cite{Rafi2007} showed that there is a unique subsurface $Y_\fl$ in the homotopy class of $ Y$ with the following properties:
\begin{itemize}
 \item $Y_\fl$ has locally geodesic boundary for the flat metric. For each boundary curve of $Y$ there is a unique flat geodesic representative in $Y_\fl$. $Y_\fl$ might be degenerated, so it might be a graph. 
\item Each free homotopy class of a simple closed curve $\alpha$ which can be homotoped into $Y$ contains a length minimizing geodesic representative for the flat metric in $Y_\fl$.
\end{itemize}
If $Y$ is not a topological pair of pants, then the constant $\lambda(Y)>0$ is the flat length of the shortest essential non-peripheral simple closed local geodesic on $Y_\fl$ and if $Y$ is a topological pair of pants, $\lambda(Y)>0$ is defined as the maximal flat length of boundary components.
\begin{Prop}\label{Prprafithinthick} 
For some $c>0$, which only depends on the topology of $X$ and for each non-peripheral simple closed curve $\alpha$ in $Y$ it follows:
 $$c^{-1}\lambda(Y) l_\sigma([\alpha])<l_\fl([\alpha])<c\lambda(Y) l_\sigma([\alpha])$$
Moreover, the diameter of $Y_\fl$ is comparable to $\lambda(Y)$.
$$c^{-1}\lambda(Y) \leq \diam(Y_\fl) \leq c\lambda(Y)$$ 
\end{Prop}
\begin{proof}
\cite{Rafi2007}
\end{proof}
Denote by $F_{i,j}$ again the subadditive process as in the previous section and $F:=\lim\limits_{T\rightarrow \infty} T^{-1} F_{0,T}$ the measurable limit. 
\begin{Thm}\label{Thmentrovsrafi}
For a closed flat surface $S=(X,d_\fl)$ and for the hyperbolic metric $\sigma$ let $Y$ be a connected component of $X_>$.\\
Then there exists a constant $A>0$ which only depends on the topology of $X$ such that
$$ F\geq A\cdot\lambda(Y).$$
\end{Thm}
\begin{proof}
Recall that $\tau: T^1X \to X$ is the projection of the unit tangent bundle, whereas $\pi:\tilde{X} \to X$ is the projection of the universal cover.\\
Assume first that the connected component $Y\subset X_>$ is not a topological pair of pants. There exist non-peripheral intersecting simple closed local geodesics $\alpha_\sigma, \beta_\sigma\subset Y$ whose hyperbolic lengths are bounded from above and below by some uniform constant. Denote by $\alpha_\fl$ resp. $\beta_\fl$ the corresponding length minimizing representatives in the free homotopy class and define $m:=\frac{2c\cdot l_\fl(\beta_\fl)}{\lambda(Y)}+4$ which is bounded by constants which only depend on the topology of $X$.\\
There is a hyperbolic convex collar $C_{\alpha}$ around $\alpha_\sigma$ with the following properties:
\begin{itemize}
\item The hyperbolic length of the shortest arc in $C_\alpha$ which connects different boundary components has a universal positive upper and lower bound which only depend on the length of $\alpha_\sigma$.
\item The hyperbolic area of $C_\alpha$ is bounded from below by some positive constant which only depends on the topology of $X$ as well.
 \item There exists some $s_0>0$, which depends on $m$ and on the hyperbolic length of $\alpha_\sigma$, so that each hyperbolic geodesic $g_\sigma:[0,s_0] \to C_\alpha$, of length at least $s_0$ intersects $\beta_\sigma$ at least $m$ times. 
\end{itemize}

The \textit{direction collar} $\mathit{CD}_{\alpha}\subset T^1 C_\alpha$ is defined as follows: \\
At each point $x \in C_\alpha$ choose a maximal set of directions $I_x\subset T^1_xX$ so that for each $v\in I_x$ the hyperbolic geodesic $g_\sigma:[0,s_0] \to X,g_\sigma:t\mapsto \tau(g_t(v))$ is entirely contained in $C_\alpha$, see Figure \ref{fig:rafi}.
\begin{figure}[htbp]
 \centering
 \fbox{
\includegraphics[width=0.6\textwidth]{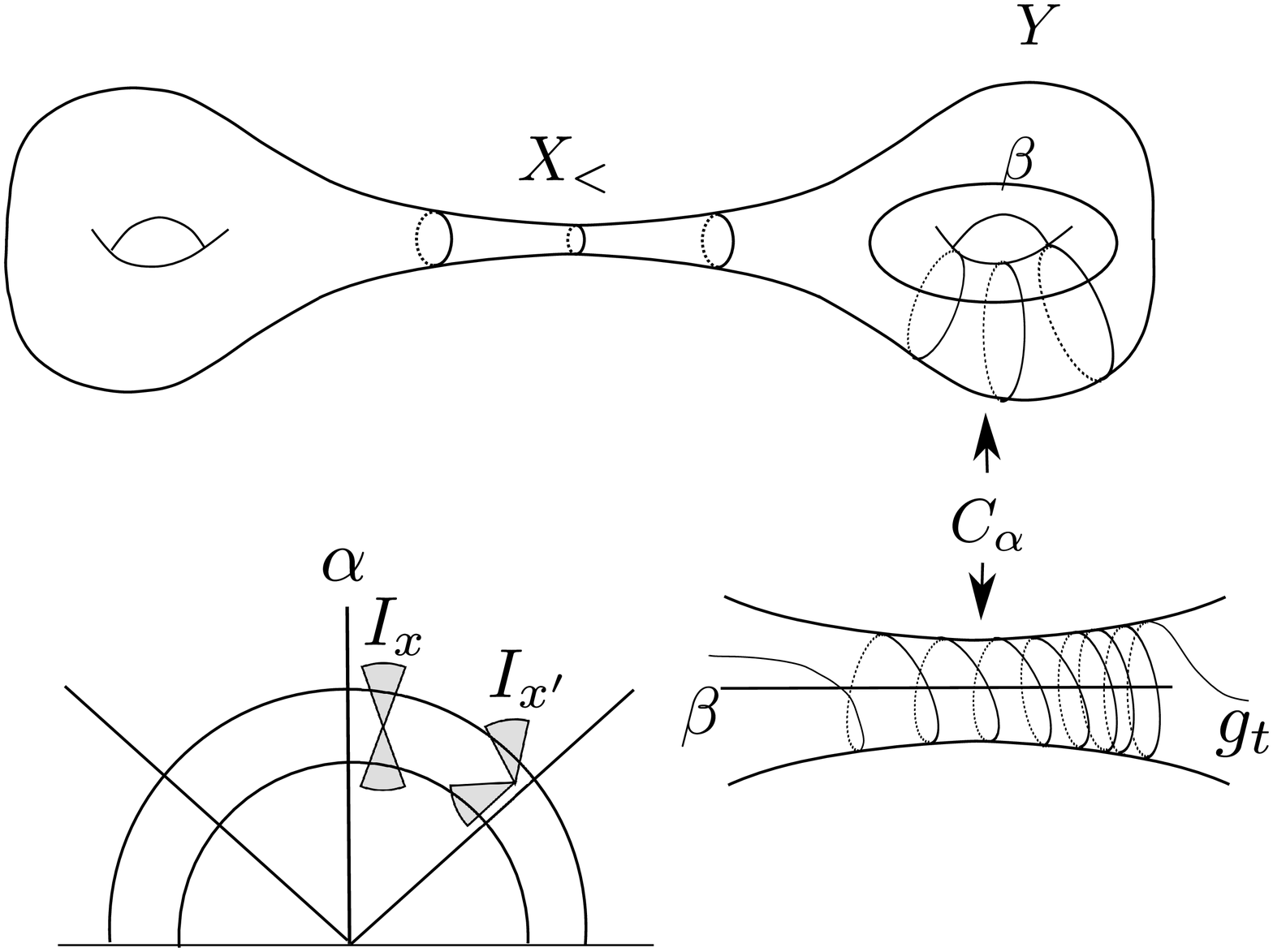}
 }
 \caption{One sees the configuration of $\alpha$ and $\beta$ on the surface. Uniform $C_\alpha$ such that the identification map is $z\mapsto az,a>0$. At each point in $C_\alpha$ choose two intervals of directions $I_x$, such that the geodesic flow in this direction twists a fixed amount of times in the collar $C_\alpha$.}
 \label{fig:rafi}
\end{figure}
The Lebesgue measure of each fiber $I_x$ has a positive lower bound which depends on $s_0$ and the length of $\alpha_\sigma$, consequently on the topology of $X$. So the volume of $\mathit{CD}_\alpha$ has a positive lower bound which only depends on the topology of $X$.\\ 
By the following Lemma one can use $\mathit{CD}_\alpha$ to estimate the quotient of hyperbolic and flat lengths of geodesic arcs. 
\pagebreak
\begin{Lem}\label{Lemcyldirmakelength}
Let $s_0$ be as above. For a geodesic arc for the hyperbolic metric $g_\sigma: [0,T]\to X$ let $t_j>0, j=1,\ldots, n_{\mathit{CD}} $ which satisfy the following properties:
\begin{enumerate}
\item $g_\sigma'(t_j) \in \mathit{CD}_\alpha$
 \item $t_{j}+s_0 < t_{j+1}< T-2s_0$ 
\end{enumerate}
Then there is a constant $c>0$, which only depends on the topology of $X$ and $k>0$, which depends on the flat metric but not on $g_\sigma$ so that the length of flat geodesic representative $g_\fl\in [g_\sigma] $ can be estimated by
$$l_\fl(g_\fl) \geq c^{-1} \lambda(Y)(n_{\mathit{CD}}-2k).$$
\end{Lem}
\begin{proof}
For $n_{\mathit{CD}}\leq 2k$ this is obvious, so assume that $n_{CD}>2k$.\\
The main line of the argument is the following: The geodesic subarc $\left.\tilde{g}_\sigma\right|_{[t_i,t_i+s_0]}$ twists in the collar  $C_\alpha$ for each $t_i$. One expects that there is a corresponding subarc of $g_\fl$ with the same behavior. The length of this flat subarc is comparable to the length of $\alpha_\fl$ multiplied with the number of twists. The main difficulty is to synchronize the behavior of the flat and the hyperbolic geodesic $g_*, *=d_\fl,\sigma$.\\
Choose lifts $\tilde{g}_*$ of $g_*,*=_\fl,\sigma$ on the universal cover with common endpoints $\tilde{x},\tilde{y}$and let $\tilde{\beta}_{\sigma,i},i=1,\ldots, n_{int}$ be the complete lifts of $\beta_{\sigma}$ which intersect $\tilde{g}_\sigma$ and which are ordered by their distance to $\tilde{x}$. By the Collar Lemma the geodesic lines $\tilde{\beta}_{\sigma,i}, \tilde{\beta}_{\sigma,i+1}$ are disjoint and their distance is bounded from below by some uniform positive constant.
Let $\tilde{\beta}_{\fl,i}$ be the complete lifts of $\beta_\fl$ which share their endpoints at infinity with $\tilde{\beta}_{\sigma,i}$. $\tilde{\beta}_{\fl,i}$ is a $L$-quasi-geodesic in the Poincar\'{e} disc and so of bounded Hausdorff distance $H(L)$ to $\tilde{\beta}_{\sigma,i}$. 
Whereas each $\tilde{\beta}_{\sigma,i}$ separates the endpoints $\tilde{x}$ and $\tilde{y}$, the lines $\tilde{\beta}_{\fl,i}$ might not separate. But, since the lines $\tilde{\beta}_{\sigma,i}$ are of fixed distance and the Hausdorff distance of $\tilde{\beta}_{\sigma,i}$ and $\tilde{\beta}_{\fl,i}$ is at most $H(L)$, there is a bound $k(L)>0$ such that $\tilde{\beta}_{\sigma,i}$ separates $\tilde{x}$ and $\tilde{y}$ for all $k(L)\leq i\leq n_{CD}-k(L)$. \\
The geodesic lines $\tilde{\beta}_{\fl,i}, \tilde{\beta}_{\sigma,i}$ admit a coarse synchronization of the geodesics $\tilde{g}_\fl, \tilde{g}_\sigma$ in the following way:
To each subarc $\tilde{b}_{\sigma} \subset \tilde{g}_\sigma$ that connects two lines $\tilde{\beta}_{\sigma,i},\tilde{\beta}_{\sigma,j},~ k<i,j<n_{int}-k$ we associate the shortest subarc of $\tilde{b}_{\fl} \subset\tilde{g}_\fl$ that connects two lines $\tilde{\beta}_{\fl,i},\tilde{\beta}_{\fl,j}$. After projecting the subarcs to the base surface they can be closed up with a piece of $\beta_\sigma$, resp. $\beta_\fl$ to closed curves which are in same free homotopy classes up to attaching multiples of $\beta$. \\ 
For $k\leq j \leq n_{Cd}-k$ there is a subarc $\tilde{b}_{j, \sigma} \subset\left.\tilde{g}_\sigma\right|_{[t_j,t_j+s_0]}$ that intersects $m$ lifts $\tilde{\beta}_{\sigma,i_j}\ldots , \tilde{\beta}_{\sigma,i_j+m}$.For different $j$ the associated subarcs $\tilde{b}_{j,\fl}$ are disjoint up to endpoints and therefore
$$l_\fl(\tilde{g}_{\fl})\geq \sum_{j=k}^{n_{CD}-k} l_\fl(\tilde{b}_{j,\fl}). $$
It remains to show that each arc $\tilde{b}_{j,\fl}$ is of length at least $c^{-1}\lambda(Y)$.
 Since $\pi(\tilde{b}_{j,\sigma})$ twists in the collar $C_\alpha$, one closes $\pi(\tilde{b}_{j,\fl})$ up along a piece of $\beta_\fl$ to a closed curve in the free homotopy class $\left[\alpha^{m'}\beta^{l}\right], |m'-(m-2)| \leq 1 $ which is not necessarily simple.
Let $\overline{b_{j,\fl}}$ be the flat geodesic representative of the free homotopy class $\left[\alpha^{m'}\beta^{l}\right]$. Remove a subarc of $\overline{b_{,\fl}}$ so that the resulting arc is a loop in the free homotopy class $\left[\alpha^{m'}\right]$ or $\left[\alpha^{m'} \beta\right]$.
Since $l_\fl(\overline{b_{j,\fl}}) \geq l_\fl( [\alpha^{m'}])-l_\fl( [\beta])=m'l_\fl(\alpha_\fl)-l_\fl( \beta_\fl) $
 one concludes that
$$l_\fl\left(\overline{b_{j,\fl}} \right) \geq (m-3) l_\fl(\alpha_\fl)-l_\fl(\beta_\fl).$$
By Proposition \ref{Prprafithinthick} there is a uniform constant $c>0$ so that
$$l_\fl(b_{j,\fl})\geq l_\fl\left(\overline{b_{j,\fl}}\right)-l_\fl(\beta_{\fl}) \geq c^{-1}\lambda(Y)(m-3)-2l_\fl(\beta_\fl)\geq  c^{-1}\lambda(Y).$$
 \end{proof}
We return to the proof of Theorem \ref{Thmentrovsrafi}. For a $\mathfrak{m}$-typical point $v \in T^1X$ there is some bound $T_0>0$ so that for all $T>T_0$ the geodesic $g_t(v), 0<t<T$ spends at least some proportional amount of time $rT$ in $CD_\alpha$where the proportion $r$ depends only on the area of $CD_\alpha$ hence on the topology of $X$. There are at least $\frac{rT}{s_0}$ times $t_j$ so that $g_{t_j}(v)\in CD_\alpha, 0<t_j<t_{j+1}-s_0<T-2s_0$. \\
By Lemma \ref{Lemcyldirmakelength} 
$$T^{-1}F_{0,T}(v)\geq T^{-1} c^{-1}\lambda(Y)\left( \frac{rT}{ s_0}-2k\right)= c^{-1}\lambda(Y)\left( \frac{r}{s_0}-\frac{2k}{T}\right). $$
where $c,s_0,r>0$ only depend on the topology of $X$ and $k>0$ is independent of $T$.\\
It remains to show the Theorem in the case that $Y$ is a pair of pants. Let $\alpha_\sigma, \beta_\sigma$ be the intersecting hyperbolic geodesics as in Figure \ref{figpant} whose lengths are bounded from above by some uniform constant.
\begin{figure}[htbp]
 \centering
 \fbox{
 \includegraphics[width=0.3\textwidth]{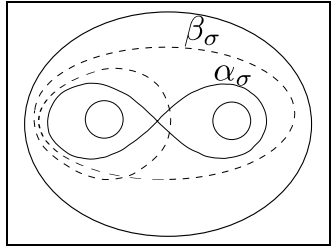}
 }
 \caption{The geodesics $\alpha_\sigma,\beta_\sigma$ on a pair of pants which is homeomorphic to a disc with two holes.}
 \label{figpant}
\end{figure}
Slightly homotope the flat geodesic representative $\alpha_\fl\in [\alpha_\sigma]$ such that it has exactly one point of self-intersection and the length only changes by some small factor and decompose the homotoped curves in two simple closed curves $\alpha_1$,$\alpha_2$ which are both in the homotopy class of boundary components. The flat length of the geodesic representative $\alpha_{1,q}, \alpha_{2,q}$ is at most $\lambda(Y)$ and by Theorem \ref{Prprafithinthick} the diameter of $Y_\fl$, with respect to the flat metric, is less than $c\lambda(Y)$ for some constant uniform constant $c$, so $l_\fl(\alpha_\fl)\leq 2\lambda(Y)(1+c)$.\\
 Analogously the flat length of the geodesic representative $\beta_\fl$ is uniformly bounded. Choose the $4$-sheeted cover $\pi:X' \to X$ of $X$, so that up to passing through $\alpha_\sigma, \beta_\sigma$ twice, the lifts $\alpha_\sigma',\beta_\sigma'$ are simple closed and intersecting. Lift the metrics to $X'$. 
 The remaining part of the argument is analogous to the one in the first case:\\
 Since the length of $\alpha_\sigma'$ is twice the length $\alpha_\sigma$, there exists some hyperbolic collar $C_{\alpha}\subset X'$ around $\alpha_\sigma'$ which has some definite width. As in the first case choose the direction collar $CD_{\alpha'}$ so that for each $v \in CD_{\alpha '}$ the flow $g_t(v),0<t<s_0$ remains in $C_{\alpha'}$ and intersects $\beta'_\sigma$ a certain number of times $m$ and that the volume of $CD_{\alpha'}$ has some definite amount. The projection of the direction collar to the unit tangent bundle of the base surface $T^1X$ has hyperbolic measure at least one fourth the hyperbolic measure of $CD_{\alpha '}$.\\
 For $v \in\pi(CD_{\alpha'})\subset T^1X$, choose a lift of the flow $\tau(g_t(v)),0<t<s_0 $ to $T^1Y'$. The lifted geodesic winds through the collar $C_{\alpha'}$ and can be closed up with a piece of $\beta'_\sigma$ to a closed curve which is freely homotopic to $\alpha'^{m'}\beta'^{l'}$. The situation on the base surface is equal. For $v \in \pi(CD_{\alpha'})$ the geodesic flow $\tau(g_t(v)),0<t<t_0$ remains in $\pi(C_{\alpha'})$ and intersects $\beta_\sigma$ at least $m$ times. One closes up the piece with $\beta_\sigma$ and obtains a closed curve which is freely homotopic to $ \alpha^{m'}\beta^l,l\in \N, |m'-(m-2)|<1 $.\\ 
 The remaining part of the argument is analogous to the one in the first case.
\end{proof}

\section{Typical boundary points}\label{sectasymprays}
In this section, we investigate the isometric universal cover $\pi: \tilde{S} \to S$ of a closed flat surface $S$ with $\Gamma$ the group of deck transformations.\\
To study geodesic rays in $ \tilde{S}$, recall that $\nu_{\tilde{x}}, \tilde{x} \in \tilde{S}$ is the Patterson-Sullivan measure on the Gromov boundary $\partial \tilde{S}$. For a point $\tilde{x}\in \tilde{S}$ and $U\subset \tilde{S}$ the boundary shadow $\partial sh_{\tilde{x}}(U) \subset \partial\tilde{S} $ was defined in Section \ref{sectgrhyp}. 
The following statements are needed.
\begin{Lem}\label{LemShadSing}
Let $\tilde{\varsigma}\in \tilde{S}-\tilde{x}$ be a singularity. Then the shadow $\partial sh_{\tilde{x}}(\tilde{\varsigma})$ contains an open subset of the boundary. 
\end{Lem}
 \begin{proof}
 Let $[\tilde{x}, \tilde{\varsigma}]$ be the geodesic connecting $\tilde{x}$ with $\tilde{\varsigma}$. As the cone angle at $\tilde{\varsigma}$ is at least $3\pi$, in a standard neighborhood of $\tilde{\varsigma}$ one finds a small open subset $\tilde{U}$ so that for any $\tilde{y} \in \tilde{U}$ the geodesic which connects $\tilde{x}$ to $\tilde{y}$ passes through $\tilde{\varsigma}$. Observe that $\partial sh_{\tilde{x}}(\tilde{\varsigma})\supset \partial sh_{\tilde{x}}(\tilde{U})$.
\end{proof}

\begin{Prop}\label{Propinfsing}
For $\tilde{x} \in \tilde{S}$ let $B_{\tilde{x}} \subset \partial \tilde{S}$ be the set of points $\eta \in \partial \tilde{S}$ such that $[\tilde{x}, \eta]$ passes through finitely many singularities. Then $\nu_{\tilde{x}}(B_{\tilde{x}})=0$. 
 \end{Prop}
\begin{proof}
Let $\Sigma_{\tilde{S}} \subset \tilde{S}$ be the set of singularities and let $B_{\tilde{x},0} \subset B_{\tilde{x}}$ be the set of points $\eta\in \partial \tilde{S}$ such that the connecting geodesic $[\tilde{x}, \eta]$ does not pass through any singularity. \\
Let $A_L\subset \Sigma_{\tilde{S}},L>0$ be the set of singularities $\tilde{\varsigma}$ of distance at most $L$ to $\tilde{x}$ so that $[\tilde{x}, \tilde{\varsigma}]$ does not pass through any other singularity. \cite{Masur1990} showed that $|A_L|$ grows quadratically in $L$.
Let $k\pi$ be the cone angle at $\tilde{x}$. On the circle of directions at $\tilde{x}$ fix a clockwise ordering and choose a base direction $\theta_0$ with the property that the geodesic in direction $\theta_0$ hits a singularity. To each point $\tilde{y}\not=\tilde{x}$ in the compactification of $\tilde{S}$ associate the angle $\vartheta(\tilde{y}) \in [0,k\pi)$ between $[\tilde{x}, \tilde{y}]$ and $\theta_0$ at $\tilde{x}$. Order the points $\tilde{\varsigma}_i\in A_L,i=1, \ldots, n $ by their angle $\vartheta$. By Lemma \ref{lemscdense}, the set of directions of singularities is dense, so after enlarging $L$ we may assume that 
\begin{eqnarray*}
0&<&\vartheta(\tilde{\varsigma}_{i+1})-\vartheta(\tilde{\varsigma}_{i})\leq \pi/3~\forall i\leq n-1\\
&&k \pi -\vartheta(\tilde{\varsigma}_{n})\leq \pi/3
 \end{eqnarray*}
As in Figure \ref{fignosaddleconnection} decompose $B_{0, \tilde{x}}$ into sets
\begin{eqnarray*}
S^L_i &:=&\{\eta \in B_{0, \tilde{x}}| \vartheta(\tilde{\varsigma}_i)< \vartheta(\eta) < \vartheta(\tilde{\varsigma}_{i+1})\}\\
 S^L_n &:=&\{\eta \in B_{0, \tilde{x}}| \vartheta(\tilde{\varsigma}_n)< \vartheta(\eta)\}
\end{eqnarray*}
\begin{figure}[htbp]
 \centering
 \fbox{
\includegraphics[width=0.6\textwidth]{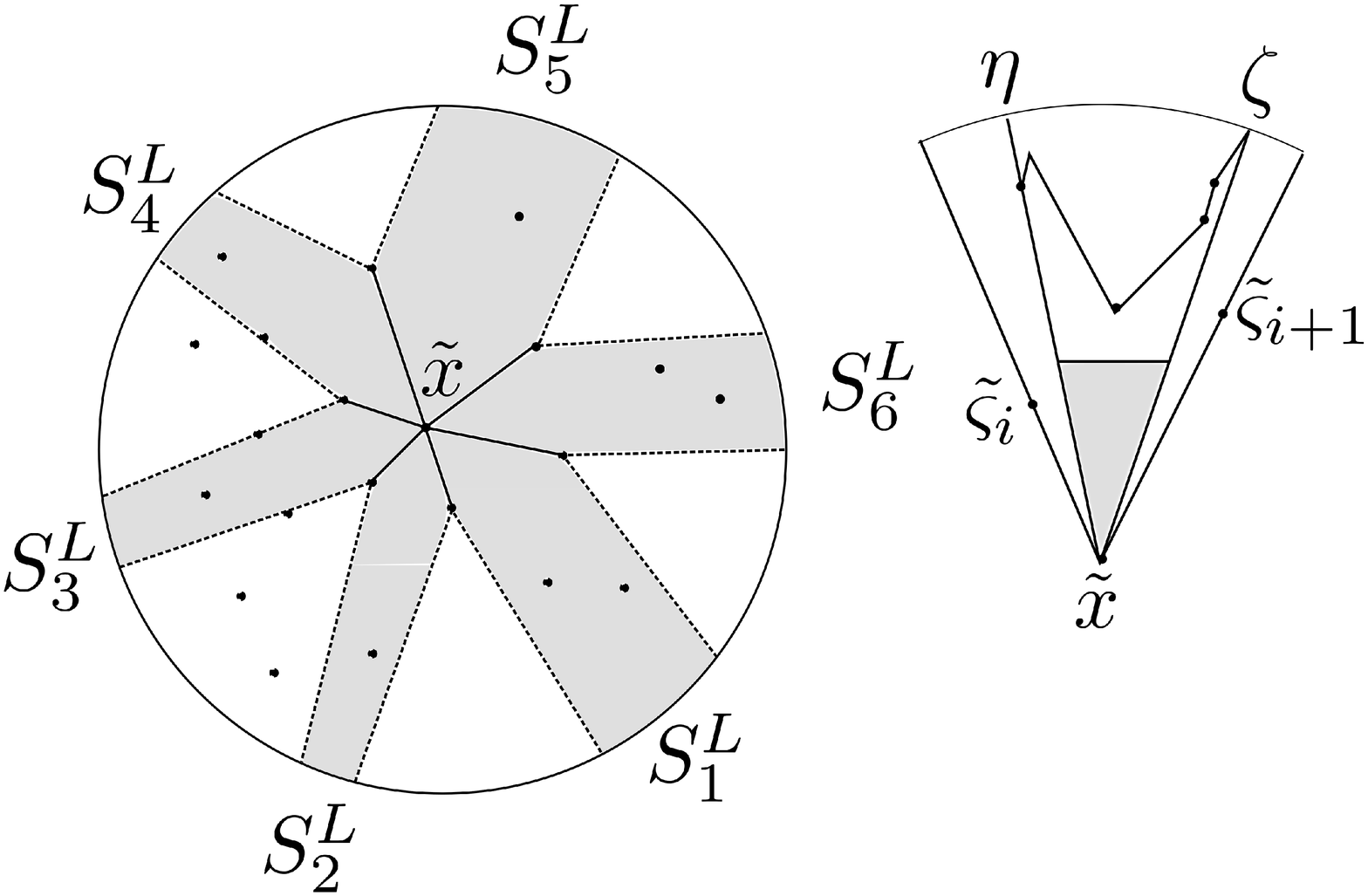}
 }
 \caption{On the left cut out the shadows of all singularities of distance at most $L$. On the right, the triangle does not contain a singularity and therefore it is euclidean.}
 \label{fignosaddleconnection}
\end{figure}
To estimate the diameter of $S^L_i$ with respect to the Gromov metric $d_{\infty,\tilde{x}}$, let $\eta, \zeta\in S^L_i,~\vartheta(\eta)\leq \vartheta(\zeta)$ and consider the triangle with vertices $\tilde{x},[\tilde{x}, \eta](L),[\tilde{x}, \zeta](L)$. This triangle is isometric to a euclidean triangle, and the inner angle at $\tilde{x}$ is less than $\pi/3$. 
Therefore we can estimate the Gromov product 
 $$([\tilde{x}, \eta](L) \cdot[\tilde{x}, \zeta](L))_{\tilde{x}} > L/3-2\delta \Rightarrow (\eta\cdot \zeta)_{\tilde{x}}\geq L/3-2\delta\Rightarrow d_{\infty,\tilde{x}}(\eta, \zeta)\leq \xi^{-L/3+4\delta}. $$ 
 So for each $L$ the sets $\bigcup S_i^L,i=1,\ldots n(L)$ cover $B_{\tilde{x},0}$. The diameter of $S^L_i$ is bounded from above by a function $\diam(L)$. Since $\diam(L)$ decreases exponentially in $L$ and $n(L)$ grows polynomially in $L$, the Hausdorff dimension of $B_{\tilde{x},0}$ has to vanish. As the Hausdorff dimension of the whole boundary is positive, $B_{\tilde{x},0}$ has measure zero with respect to the measure class $\nu_{\tilde{y}}, \tilde{y} \in \tilde{S}$. As
$$B_{\tilde{x}} \subset \bigcup\limits_{\tilde{\varsigma} \in \Sigma_{\tilde{S}}} B_{\tilde{\varsigma},0}\cup B_{\tilde{x},0}$$ 
and there are only countably many singularities, $\nu_{\tilde{x}}(B_{\tilde{x}})=0$
\end{proof}
Whereas geodesic rays in the Poincar\'{e} disc with the same boundary endpoint converge towards each other there are geodesic rays in $\tilde{S}$ with the same endpoint which do not converge, as the following examples shows: \\ 
A \textit{maximal flat strip} in a metric space is an isometric embedding of $(0,r)\times \R$ which cannot be extended. \cite{Masur1986} showed that a the universal cover $\tilde{S}$ of closed flat surface $S$ contains countable many $\Gamma$-orbits of maximal flat strips. 
Choose distinct geodesic lines $\tilde{\alpha}_i,~i=1,2$ in the same flat strip. Then $\tilde{\alpha}_i$ have finite Hausdorff distance but do not converge towards each other.\\
We will show that for a $\nu_{\tilde{x}}$-typical boundary point $\eta$, any two geodesic rays which converge to a $\eta$ eventually coincide. The argument is inspired by \cite{Duchin2010}.\\ 
A parametrized geodesic line $\tilde{\alpha}:\R\to \tilde{S}$ decomposes $\tilde{S} $ in two connected components $\tilde{S}^{\pm}$. 
For $\tilde{\alpha}(t)$ choose a standard neighborhood $\tilde{U}$ and an orientation such that $\vartheta^+(\tilde{\alpha}(t))$ or $\vartheta^-(\tilde{\alpha}(t))$ respectively measures the flat angle 
$$\angle_{\tilde{\alpha}(t)}(\left.\tilde{\alpha} \right|_{[t-\epsilon,t]}, \left.\tilde{\alpha} \right|_{[t,t+\epsilon]})$$
 inside of $\tilde{U}\cap \tilde{S}^+$ or $\tilde{U}\cap \tilde{S}^-$ respectively. \\ 
For a geodesic ray $\tilde{r}:[0,\infty) \to \tilde{S}$ choose an extension to a geodesic line and define $\vartheta^+(r(t)),~t>0$  analogously.
\begin{Def}
 A geodesic ray $\tilde{r}$ is called quasi-straight if there is some $i \in \pm$ so that 
$$\sum\limits_{t>0}(\vartheta^i(\tilde{r}(t))-\pi)<\infty.$$
A boundary point $\eta \in \partial \tilde{S}$ is \textit{quasi-straight} if there is a quasi-straight geodesic ray which tends to $\eta$.
\end{Def}
Denote by $\qstraight \subset \partial\tilde{S}$ the $\Gamma$-invariant set of quasi-straight boundary points.\\
Since the choice of $\vartheta^+$ is arbitrary, we may assume that for any quasi-straight geodesic ray the angle $\vartheta^+$ is bounded.
\begin{Prop}\label{Proptangeod}
Assume that $\tilde{r}_1,\tilde{r}_2$ are geodesic rays in $\tilde{S}$ with the same endpoint $\eta \in \partial \tilde{S}$. If $\tilde{r}_2$ is not quasi-straight, then $\tilde{r}_1$ and $\tilde{r}_2$ eventually coincide
 \end{Prop}
\begin{proof}
Extend $\tilde{r}_2$ to a geodesic line $\tilde{\alpha}$. By uniqueness of geodesic rays, $\tilde{\alpha}$ and $\tilde{r}_1$ are either disjoint or eventually coincide. Assume that they are disjoint and let $\tilde{S}^+$ be the component of$\tilde{S}-\tilde{\alpha}$ that contains $\tilde{r}_1$. Denote by $\tilde{c}_t$ the geodesic connecting $\tilde{r}_1(t)$ with $\tilde{\alpha}(t)$ and consider the rectangle $R_t,t>0$ with sides $\left.\tilde{\alpha} \right|_{[0,t]},\left.\tilde{r}_1 \right|_{[0,t]},\tilde{c}_0,\tilde{c}_t$. Choose $s_0>0$ that $\tilde{\alpha}([s_0,\infty])$ is disjoint from $c_0$ and choose $t_0>s_0$ that
$$\sum\limits_{s_0<s<t_0}(\vartheta^+(\tilde{\alpha}(s))-\pi)>2\pi.$$
By the Gauss-Bonnet formula, for any rectangle in $\tilde{S}^+$ with interior $P$, the boundary of $P$ does not contain the whole side $\left. \alpha \right|_{[s_0,t_0]}$. So, the geodesic $\tilde{c}_{t_0}$ intersects $\tilde{\alpha}(t_0-\epsilon)$ and therefore $\tilde{c}_{t_0}$ contains $\left. \tilde{\alpha} \right|_{[t_0-\epsilon,t_0]}$. As a consequence one can extend $\tilde{c}_{t_0}$ along $\tilde{\alpha}$ to a geodesic ray with the endpoint $\eta$. 
This is a contradiction, as $\tilde{c}_{t_0}$ issued from $\tilde{r}_1(t_0)$ and there is only one geodesic ray connecting $\tilde{r}_1(t_0)$ with $\eta$. 
\end{proof}
Consequently a boundary point $\eta$ is quasi-straight if and only if each ray which converges to $\eta$ is quasi-straight.
\begin{Prop}\label{propquasistraight}
 The set of quasi straight boundary points is of vanishing $\nu_{\tilde{x}}$-measure.
\end{Prop}
The following technical Lemma is needed.
\begin{Lem}\label{Leminterlength}
For the flat universal cover of a closed flat surface $\pi:\tilde{S} \to S$ there is some positive function $R: \R_{>0} \to \R_{>0}$ such that the following holds:\\
For a singularity $\tilde{\varsigma}\in \tilde{S}$ and a closed interval of directions $I$ of length $l>0$ at $\tilde{\varsigma}$ denote by $S_I$ the set of boundary points $\eta \in \partial \tilde{S}$ such that the direction of $[\tilde{\varsigma}, \eta]$ is contained in $I$\\
Then $S_I$ is Borel, and the size of $S_I$ is bounded from below the size of $R(l)$:
$$\nu_{\tilde{\varsigma}}(S_I)\geq R(l)$$
\end{Lem}
\begin{proof}
For a standard neighborhood $U$ of $\tilde{\varsigma}$, the set of points $U_I \subset\partial U$ in direction $I$ form a closed set. Thus $S_I=\partial sh_{\tilde{\varsigma}}(U_I)$ is Borel.\\ 
Choose finitely many singularities $\tilde{\varsigma}_i\in \Sigma_{\tilde{S}},i=1\ldots n$ in $\tilde{S}$ whose $\Gamma$-orbits contain all singularities. Cover the circle of directions at $\tilde{\varsigma}_i$ with finitely many closed intervals $I_{i,j}$ of length $l/3$ and define the boundary intervals $S_{I_{i,j}}\subset \partial \tilde{S}$ as the set of points $\eta$ so that the direction of $[\tilde{\varsigma}_i, \eta]$ is contained in $I_{i,j}$. Define
$$R(l):= \min\limits_{i,j}~\nu_{\tilde{\varsigma}_i}(S_{I_{i,j}})>0$$
For every other singularity $\tilde{\varsigma}$ in the $\Gamma$-orbit of $\tilde{\varsigma}_i$ translate the covering of directions and obtain the same measure. Observe that each interval of directions $I$ of length $l$ at $\tilde{\varsigma}_i$ has to contain one of the smaller intervals of length $l/3$. 
\end{proof}

\begin{proof}[Proof of the Proposition.]
Denote by $\Sigma_{\tilde{S}}$ the singularities on $\tilde{S}$ and recall that the set of quasi-straight boundary points $\qstraight$ is defined by the property that for any $\tilde{x}\in \tilde{S}$ 
$$\sum_{t>0}(\vartheta^+([\tilde{x}, \eta](t))-\pi )< \infty,~\forall \eta \in \qstraight.$$
It suffices to consider only those quasi-straight points $\eta$ such that $ [\tilde{x}, \eta]$ passes through infinitely many singularities.
Let $A_{\tilde{x}} \subset \qstraight$ be the set of quasi-straight endpoints $\eta$ such that 
$$\sum_{t>0}(\vartheta^+([\tilde{x}, \eta](t))-\pi) \leq \pi/3.$$
As $\bigcup\limits_{\tilde{\varsigma} \in \Sigma_{\tilde{S}}} A_{\tilde{\varsigma}}$ covers the quasi-straight points up to measure $0$ it suffices to show that $A_{\tilde{x}}$ is of measure $0$. \\
For each $n$ let $A_{\tilde{x},n} \subset \Sigma_{\tilde{S}}$ be the set of singularities $\tilde{\varsigma}_n$ such that $[\tilde{x}, \tilde{\varsigma}_n]$ passes through $n$ singularities and that at each singularity $\tilde{\varsigma} \in [\tilde{x}, \tilde{\varsigma}_n]$, the smaller angle between the two rays $[\tilde{x}, \tilde{\varsigma}]$ and $[\tilde{\varsigma}, \tilde{\varsigma}_n]$ is between $ \pi$ and $4\pi/3$, see Figure \ref{figquasistr}.
\begin{figure}[htbp]
 \centering
 \fbox{
\includegraphics[width=0.4\textwidth]{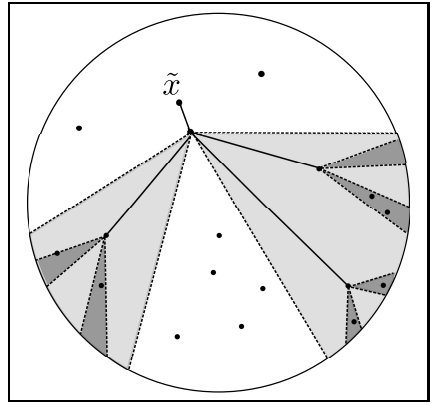}
 }
 \caption{For any singularity in $A_{\tilde{x},1}$ the further points in $A_{\tilde{x},2}$ are contained in the light gray part. For the consecutive singularity in $A_{\tilde{x},3}$ we see a splitting in two parts, hence a Cantor construction. }
 \label{figquasistr}
\end{figure}
Define the mapping 
$$\phi_n:A_{\tilde{x},n+1} \to A_{\tilde{x},n},~\tilde{\varsigma}_{n+1} \mapsto \tilde{\varsigma}_{n}$$ 
which projects $\tilde{\varsigma}_{n+1} \in A_{n+1}$ to the $n$-th singularity in $[\tilde{x}, \tilde{\varsigma}_{n+1}]$.\\ 
Consequently 
$$A_{\tilde{x}} \subset \partial sh_{\tilde{x}}(A_{\tilde{x},n+1}) \subset \partial sh_{\tilde{x}}(A_{\tilde{x},n})$$ 
 For $\tilde{\varsigma}_n\not=\tilde{\varsigma}_n'\in A_{\tilde{x},n}$ the shadows $\partial sh_{\tilde{x}}(\tilde{\varsigma}_n)$ and $\partial sh_{\tilde{x}}(\tilde{\varsigma}_n')$ are disjoint and so
$$\nu_{\tilde{x}}(\partial sh_{\tilde{x}}(A_{\tilde{x},n}))=\sum_{\tilde{\varsigma}_n \in A_{\tilde{x},n}} \nu_{\tilde{x}}(\partial sh_{\tilde{x}}(\tilde{\varsigma}_n)).$$
Fix $\tilde{\varsigma}_{n+1}\in A_{\tilde{\varsigma}_0,n+1}$ and $\tilde{\varsigma}_n:=\phi_n(\tilde{\varsigma}_{n+1})$. 
A boundary point $\eta$ is contained in $\partial sh_{\tilde{x}}(\tilde{\varsigma}_{n+1})$ only if the direction of $[\tilde{\varsigma}_n, \tilde{y}]$ is contained in two intervals of directions $I_1,I_2$ at $\tilde{\varsigma}_n$ which are both of length $\pi/3$.\\
 The set of boundary points in the whole shadow $\partial sh_{\tilde{x}}(\tilde{\varsigma}_n)$ defines an interval of direction at $\tilde{\varsigma}_n$ of length at least $\pi$, so the boundary shadow $\partial sh_{\tilde{x}}(\phi_n^{-1}(\tilde{\varsigma}_n))$ misses a subinterval of directions of length at least $\pi/3$ at $\tilde{\varsigma}_{n}$.\\ 
For the function $R$ of Lemma \ref{Leminterlength} it follows that 
\begin{eqnarray*}
&&\nu_{\tilde{\varsigma}_n}(\partial sh_{\tilde{x}}(\tilde{\varsigma}_n))\geq\nu_{\tilde{\varsigma}_n}(\partial sh_{\tilde{x}}(\phi_n^{-1}(\tilde{\varsigma}_n)))+R(\pi/3)\\
&\Leftrightarrow&1-\frac{R(\pi/3)}{\nu_{\tilde{\varsigma}_n}(\partial sh_{\tilde{x}}(\tilde{\varsigma}_n))}\geq \frac{\nu_{\tilde{\varsigma}_n}(\partial sh_{\tilde{x}}(\phi_n^{-1}(\tilde{\varsigma}_n)))}{\nu_{\tilde{\varsigma}_n}(\partial sh_{\tilde{x}}(\tilde{\varsigma}_n))}=\frac{\nu_{\tilde{x}}(\partial sh_{\tilde{x}}(\phi_n^{-1}(\tilde{\varsigma}_n))}{\nu_{\tilde{x}}(\partial sh_{\tilde{x}}(\tilde{\varsigma}_n))}
\end{eqnarray*}
The last equality is due to Lemma \ref{Lemradder}. Since $\nu_{\tilde{\varsigma}_n}(\partial sh_{\tilde{x}}(\tilde{\varsigma}_n))\leq \exp(e(S)\diam(S))$, the left term has a uniform upper bound $\lambda<1$. Therefore 
$$ \nu_{\tilde{x}}\left(\partial sh_{\tilde{x}}(A_{\tilde{x},n+1})\right)\leq \lambda\cdot\nu_{\tilde{x}}\left( \partial sh_{\tilde{x}}(A_{\tilde{x},n})\right)$$
Since 
$$A_{\tilde{x}}\subset\bigcap_{n\geq 0} \partial sh_{\tilde{x}}(A_{\tilde{x},n})$$
and since $\nu_{\tilde{y}}$ was extended to a complete measure$\nu_{\tilde{y}}(A_{\tilde{x}})=0$ for any $\tilde{y}\in \tilde{S}$.
\end{proof}

 \section{Geodesic flow} \label{sectgeodflow}
To investigate how the typical geodesic winds through $S$, we define the geodesic flow with respect to the flat metric. 
Similar constructions have already been made by Kaimanovich \cite{Kaimanovich1994} in the case of $\delta$-hyperbolic spaces with additional conditions, by Bourdon \cite{Bourdon1995} for $\Cat(-1)$ spaces and by Coornaert and Papadopoulos \cite{Coornaert1994}, \cite{Coornaert1997} for metric trees resp. graphs. 
\subsection{Construction of the geodesic flow}
We recall the necessary concepts and results.\\ 
For the flat universal cover $\pi:\tilde{S} \to S$ of a closed flat surface let $\nu_{\tilde{x}}$ be the Patterson-Sullivan measure on the compactification of $\tilde{S}$ with respect to some base point $\tilde{x}\in \tilde{S}$. We defined the quasi-straight boundary points which are of measure 0 for $\nu_{\tilde{x}}$. 
Denote by $d_{\infty, \tilde{x}}$ the family of Gromov metrics on $\partial \tilde{S}$ which satisfies 
$$(1-\epsilon(\xi))(\xi^{-(\eta\cdot \zeta)_{\tilde{x}}}) \leq d_{\infty, \tilde{x}}(\eta, \zeta)\leq \xi^{-(\eta\cdot \zeta)_{\tilde{x}}}. $$ 
 Since the flat metric is not smooth, one cannot make use of the unit tangent bundle for the geodesic flow. The space of bi-infinite parametrized unit speed geodesics plays the role of the unit tangent bundle. 
\begin{Def}
Let $\G\tilde{S}$ or $\G S$ respectively be the set of all parametrized bi-infinite unit speed geodesics $\alpha$ in $\tilde{S}$ or $S$ respectively which is endowed with the metric
$$d_{\G}(\alpha, \alpha'):=\int_{-\infty}^\infty d(\alpha(t), \alpha'(t))exp(-|t|)dt.$$
\end{Def}
 $\G \tilde{S}$ is proper, and $\G S$ is compact. The group of Deck transformations $\Gamma$ acts isometric, properly discontinuously and freely on $\G\tilde{S}$ as 
 $$\gamma:\tilde{\alpha} \mapsto \gamma(\tilde{\alpha}).$$
The canonical mapping $\G \tilde{S}/\Gamma \to \G S$ is a homeomorphism. The \textit{geodesic flow} $g_t$ acts on $\G\tilde{S}$ resp. $\G S$ as $g_t(\alpha)(s):=\alpha(t+s)$.\\
Define $\tau: \G\tilde{S} \to \partial^2 \tilde{S}-\triangle$ to be the projection of a geodesic onto its endpoints. $\tau$ is onto and equivariant with respect to $\Gamma$. $g_t$ acts on the fibers of $\tau$.\\
Fix $\tilde{x}\in \tilde{S}$ and the Patterson-Sullivan measure $\nu_{\tilde{x}}$. For $(\eta, \zeta) \in \partial \tilde{S}-\triangle$ there is a geodesic $[\eta, \zeta]$ connecting $\eta$ with $\zeta$. Fix $\tilde{y} \in [\eta, \zeta]$ and define
$$\iota_{\tilde{x}}: (\partial^2 \tilde{S}-\triangle) \to \R_+, \iota_{\tilde{x}}(\eta, \zeta):=\frac{d\nu_{\tilde{y}}}{d\nu_{\tilde{x}}}(\eta)\frac{d\nu_{\tilde{y}}}{d\nu_{\tilde{x}}}(\zeta).$$
If $\eta$ or $\zeta$ is not quasi-straight, $[\eta, \zeta]$ is unique up to reparametrization , by Proposition \ref{propquasistraight}. By Lemma \ref{Lemradder}, $\iota_{\tilde{x}}$ is independent of the choice of $\tilde{y}\in [\eta, \zeta]$. Moreover, the mapping $t\mapsto g_t([\eta, \zeta])$ defines an $\R$-parametrization of the fiber $\tau^{-1}(\eta, \zeta)$.
Define the following Borel measure on $\partial^2 \tilde{S}-\triangle$:
$$\tilde{\nu}_{\tilde{x}}:=\iota_{\tilde{x}} \cdot \nu^2_{\tilde{x}}$$
which, by Lemma \ref{Lemradder} satisfies the following properties:
\begin{itemize}
 \item $\tilde{\nu}_{\tilde{x}}$, is independent of $\tilde{x}$ and so can be abbreviated by $\tilde{\nu}$. 
\item $\tilde{\nu}$ is a $\Gamma$-invariant infinite Radon measure in the measure class of $\nu_{\tilde{x}}^2$.
\end{itemize}
For a $\tilde{\nu}$-typical point $(\eta, \zeta)\in \partial^2 \tilde{S}-\triangle$, fix a connecting geodesic $[\eta, \zeta]$ and pull back the Lebesgue measure $dt$ on $\R$ to the fiber via the parametrization of $g_t$. This fiber measure is independent of the parametrization as the transition map is a translation. 
\begin{Def}
Define the $\Gamma$- invariant Radon measure $\tilde{\mu}$ on $\G\tilde{S}$ 
 $$\tilde{\mu}:=\tilde{\nu}\times dt .$$
 \end{Def}
 For $U\subset \tilde{S}$ Borel define the Borel set $\G U:=\{\tilde{\alpha} \in \G \tilde{S}, \tilde{ \alpha}(0)\in U\}$.
\begin{Def}
A Borel fundamental domain $F\subset \tilde{S}$ of $\tilde{S}/\Gamma$ defines the Borel fundamental domain $\G F$ for $\G\tilde{S}/\Gamma$. Define the finite Radon measure $\mu$ on $\G S$ as 
$$\mu(U):= \tilde{\mu}(\pi^{-1}(U)\cap \G F).$$
\end{Def}

\begin{Rmk}
 $\mu$ is independent of $F$. One can also choose directly a Borel fundamental domain in $\G \tilde{S}$ and obtain the same measure $\mu$. On the other hand, special properties of the domain $\G F$ are needed later. 
\end{Rmk}
\begin{Prop}
 $g_t$ acts $\mu$-ergodically on $(\G S, \mu)$.
\end{Prop}
\begin{proof}
This follows from the Hopf Argument \cite{Hopf1971} which we sketch here. \\ 
Let $f:\G S \to \R$ be a continuous function with compact support. By the Birkhoff Ergodic Theorem $\lim\limits_{s\rightarrow \infty}s^{-1}\int^s_0 f(g_t)dt$ converges a.e. to a measurable $g_t$-invariant function $f^*$. \\
As for any two bi-infinite $\mu$-typical geodesics $\alpha,\alpha' \in \G S$ in the base surface there are parametrized geodesics $\alpha_1,\alpha_2 \in \G S$, such that each of the pairs
 $(\tilde{\alpha},\tilde{\alpha}_1)$, $ (\tilde{\alpha}_1,\tilde{\alpha}_{2})$, $(\tilde{\alpha}_2,\tilde{\alpha}')$ is asymptotic in positive or negative direction, $f^*(\alpha)=f^*(\alpha')$ and so $f^*$ is constant a.e. 
\end{proof}
\subsection{Typical behavior}\label{secttypi}
On a flat surface $S$, the extension of a compact geodesic arc which connects singularities is not unique. It turns out that the set of geodesics, which pass through such an arc has positive measure. Recall that $e(S)$ is the entropy of $S$. 
 \begin{Def}
For the flat universal cover $\pi:\tilde{S}\to S$ of a closed flat surface fix $\tilde{x}\in \tilde{S}$. For a singularity $\tilde{\varsigma}$ define
$$r_{\tilde{x}}(\tilde{\varsigma}):=\nu_{\tilde{\varsigma}}(\partial sh_{\tilde{x}}(\tilde{\varsigma}))= \nu_{\tilde{x}}\left(\partial sh_{\tilde{x}}(\tilde{\varsigma})\right)\exp(e(S)d(\tilde{x}, \tilde{\varsigma})).$$
\end{Def}
The equation follows from Lemma \ref{Lemradder}. We will show, that $r_{\tilde{x}}(\tilde{\varsigma})$ is nearly constant and so one can roughly identify $\nu_ {\tilde{x}}(\partial sh_{\tilde{x}}(\tilde{\varsigma}))$ with $\exp(-e(S)d(\tilde{x}, \tilde{\varsigma}))$.
\begin{Lem} \label{Lemmeasshad}
$r_{\tilde{x}}(\tilde{\varsigma})$ is bounded from above and below by a constant $C(S)>0$. 
$$C(S)^{-1} \leq r_{\tilde{x}}(\varsigma) \leq C(S) .$$
 \end{Lem}
 \begin{proof}
 Denote by $\diam (S)$ the diameter of $S$ and choose $C(S)\geq \exp(e(S)\diam(S))$. Then, the upper bound follows from Corollary \ref{Corestrad}:
 $$r_{\tilde{x}}(\varsigma)=\nu_{\tilde{\varsigma}}(\partial sh_{\tilde{x}}(\tilde{\varsigma}))\leq\nu_{\tilde{\varsigma}}(\partial \tilde{S})\leq \exp(e(S)\diam(S))$$
 On the other hand, fix some singularity $\tilde{\varsigma}_0\in \tilde{S}$. By Lemma \ref{lemscdense} and the fact that the cone angle at $\tilde{\varsigma}_0$ is at least $3\pi$ there are $4$ saddle connections $\tilde{s}_1, \ldots \tilde{s}_4$ with starting point $\tilde{\varsigma}_i$ and endpoint $\tilde{\varsigma}_0$ so that the shadows 
 $$\bigcup_{i\leq 4}\partial sh_{\tilde{\varsigma}_i}(\tilde{\varsigma}_0)=\partial \tilde{S}$$ 
 cover the boundary. 
 As $\nu_{\tilde{\varsigma}_0}(\partial \tilde{S})\geq \exp(-e(S)\diam(S))$, up to renumbering the saddle connections we can assume 
$$r_{\tilde{\varsigma}_1}(\tilde{\varsigma}_0)=\nu_{\tilde{\varsigma}_0}(\partial sh_{\tilde{\varsigma}_1}(\tilde{\varsigma}_0)) \geq \exp(-e(S)\diam(S))/4.$$ 
 By Proposition \ref{Propjoingeod} there exists a geodesic $\tilde{c}$ which first connects $\tilde{x}$ with $\tilde{\varsigma}$ and eventually passes through $\gamma(\tilde{s}_1)$ for some $\gamma \in \Gamma$ so that the length of $\tilde{c}$ is bounded 
$$l(\tilde{c})\leq C_l(S)+ l(\tilde{s}_1)+d(\tilde{x}, \tilde{\varsigma})$$
see Figure \ref{fignestshad}. 
\begin{figure}[htbp]
 \centering
 \fbox{
 
\includegraphics[width=0.3\textwidth]{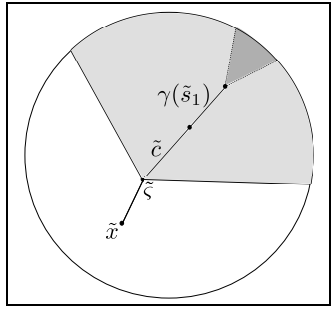}
 }

 \caption{The shadow of $\tilde{\varsigma}$ with respect to $\tilde{x}$, marked light gray, contains the shadow of the endpoint of $\gamma(\tilde{s}_1)$, here dark gray. }
 \label{fignestshad}
\end{figure}
As $\partial sh_{\tilde{x}}(\tilde{\varsigma})$ contains $\partial sh_{\tilde{x}}(\gamma(\tilde{\varsigma}_0))$ one observes
\begin{eqnarray*}
r_{\tilde{x}}(\tilde{\varsigma})&\geq&\nu_{\tilde{x}}(\partial sh_{\tilde{x}}(\gamma(\tilde{\varsigma}_0)))\exp(e(S)d(\tilde{x}, \tilde{\varsigma}))\\
&\geq & \exp(-e(S)(l(s_1)+C_l(S)+\diam(S)))/4
 \end{eqnarray*}
After enlarging $C(S)$ this completes the proof. 
\end{proof}
We return to the geodesic flow.
\begin{Thm}\label{Thmgenericbehaviour}
There is a constant $C(S)>0$ which depends on the geometry of $S$ such that the following holds: For any local geodesic $c: [0,s] \to S$ of positive length, let $c_{ext}$ be the maximal extension of $c$ with the property that the extension in unique. A typical geodesic passes through $c$ with a frequency $\varLambda$ which is bounded from above and below by 
$$C(S)^{-1}exp(-e(S)l(c_{ext}))\leq \varLambda\leq C(S) exp(-e(\tilde{S}, \Gamma_S)l(c_{ext})).$$ 
\end{Thm}
\begin{proof}
Denote by $A_c \subset \G S$ the Borel set 
$$ A_c:=\{\alpha \in \G S: \exists 0<t\leq 1: \left.g_t(\alpha)\right|_{[0,s]}=c\}.$$
As $g_t$ acts ergodically with respect to $\mu$, for any typical geodesic $\alpha \in \G S$ 
$$\lim\limits_{T\rightarrow \infty}\frac{1}{2T}\int_{-T}^{T} 1_{A_c}(g_t(\alpha)) dt= \frac{\mu(A_c)}{\mu(\G S)}.$$
Choose some lift $\tilde{c}$ of $c$ in the universal cover $\tilde{S}$ and let $F\subset \tilde{S}$ be a Borel fundamental domain of $S$ which contains $\tilde{c}(0)$ in its interior.\\
Choose $\epsilon>0$ such that 
\begin{eqnarray*}
&&B_{\tilde{c}(0)}(\epsilon) \subset \interior{F}\\
&& d(\gamma(\tilde{c}(0)),F) >\epsilon, \forall \gamma \in \Gamma-id
\end{eqnarray*}
Define
\begin{eqnarray*}
&&A_{c,\epsilon}:=\{\alpha\in \G S:\left. g_t(\alpha)\right|_{[0,s]}=c,0 <t\leq \epsilon\}\\
&&\tilde{A}_{\tilde{c},\epsilon}:= \{\tilde{\alpha} \in G\tilde{S}:\left.g_t(\tilde{\alpha})\right|_{[0,s]}=\tilde{c},0<t\leq\epsilon\}
\end{eqnarray*}
By construction
$$\mu(A_c)=\frac{1}{\epsilon} \mu(A_{c,\epsilon}) = \frac{1}{\epsilon}\tilde{\mu}(\tilde{A}_{\tilde{c},\epsilon}).$$ 
Choose $\tilde{ x}:=\tilde{c}(s/2)$ the midpoint of $\tilde{c}$ and observe that 
$$\tau(\tilde{A}_{\tilde{c},\epsilon})= \partial sh_{\tilde{x}}(\tilde{c}(0))\times \partial sh_{\tilde{x}}(\tilde{c}(s)).$$
Furthermore, for each pair of points $(\eta, \zeta) \in \tau(\tilde{A}_{\tilde{c},\epsilon})$ there exists a unique connecting geodesic $[\eta, \zeta]$ which coincides with $\tilde{c}$ on the interval $[0,s]$. Consequently 
$$g_t([\eta, \zeta]) \in \tilde{A}_{\tilde{c},\epsilon} \Leftrightarrow t \in [0, \epsilon] .$$
The Lebesgue measure of the intersection of the fiber with $\tilde{A}_{\tilde{c},\epsilon}$ is $\epsilon$ and so
$$\mu(A_{c})= \tilde{\nu}(\partial sh_{\tilde{x}}(\tilde{c}(0))\times \partial sh_{\tilde{x}}(\tilde{c}(s)))=\nu_{\tilde{x}}(\partial sh_{\tilde{x}}(\tilde{c}(0)))\cdot\nu_{\tilde{x}} (\partial sh_{\tilde{x}}(\tilde{c}(s))).$$ 
where $\nu_{\tilde{x}}$ is the Patterson Sullivan measure. \\
If the endpoint $\tilde{c}(s)$ is regular then there is locally only one possibility of extending $\tilde{c}$ in positive direction so that it remains geodesic. Extend $\tilde{c}$ as far as possible in positive as well as negative direction as long as the extension is unique, i.e. until the extension either hits a singularity or tends to infinity. \\
If the extended geodesic $\tilde{c}_{ext}$ is infinite, one of the two factors in the product shadow is a point and therefore the product has $\tilde{\nu}$-measure zero.\\ 
Thus, assume that the extension is finite. Parametrize the extended geodesic $\tilde{c}_{ext}:[-s_1,s_2] \to S$ so that $\left.\tilde{c}_{ext}\right|_{[0,s]}=\tilde{c}$. By uniqueness of the extension 
\begin{eqnarray*}
&&\partial sh_{\tilde{x}}(\tilde{c}(0))= \partial sh_{\tilde{x}}(\tilde{c}_{ext}(-s_1))\\
&&\partial sh_{\tilde{x}}(\tilde{c}(s))= \partial sh_{\tilde{x}}(\tilde{c}_{ext}(s_2)) 
\end{eqnarray*}
Both endpoints of $\tilde{c}_{ext}$ are singularities, so by Lemma \ref{Lemmeasshad} there is a universal constant $C(S)$ such that 
\begin{eqnarray*}
C(S)^{-1}exp(-e(S)(s_1+s/2))\leq &\nu_{\tilde{x}}(\partial sh_{\tilde{x}}(\tilde{c}_{ext}(-s_1)))&\leq C(S) exp(-e(S)(s_1+s/2))\\
C(S)^{-1}exp(-e(S)(s_2-s/2))\leq &\nu_{\tilde{x}}(\partial sh_{\tilde{x}}(\tilde{c}_{ext}(s_2)))&\leq C(S) exp(-e(S)(s_2-s/2))
\end{eqnarray*}
Altogether 
$$C(S)^{-2}exp(-e(S)l(\tilde{c}_{ext}))\leq \mu(A_{c}) \leq C(S)^2exp(-e(S)l(\tilde{c}_{ext}))$$ 
So the frequency a $\mu$-typical geodesic $\alpha \in \G S$ enters $c$ is proportional to
$$exp(-e(S)l(c_{ext})).$$
\end{proof}
 \bibliographystyle{alpha}	
 \bibliography{datab}
\noindent
MATHEMATISCHES INSTITUT DER UNIVERSIT\"AT BONN\\
ENDENICHER ALLEE 60,\\
53115 BONN, GERMANY\\
e-mail: klaus@math.uni-bonn.de
\end{document}